\documentclass[a4paper,11pt]{amsart}
\usepackage{amssymb,epsfig,amsxtra, amsmath}
\usepackage{amscd} 
\usepackage{xypic}
\usepackage[mathscr]{eucal} 
\usepackage{a4}
\input{xy}
\xyoption{all}

\newtheorem{thm}{Theorem}[section]

\newtheorem{lemma}[thm]{Lemma}

\renewcommand{\proofname}{Proof}
\newtheorem{proposition}[thm]{Proposition}

\theoremstyle{definition}

\newtheorem{definition}[thm]{Definition}

\def\eps{\varepsilon}
\def\mult{\operatorname{mult}}

\def\RR{{\mathbb R}}
\def\ZZ{{\mathbb Z}}

\def\PP{{\mathbb P}}

\def\O{{\mathcal O}}
\def\F{{\mathcal F}}

\def\Pic{\operatorname{Pic}}

\def\ag{\`a}

\begin{document}

\title[Seshadri constants of $K3$ surfaces]{ 
Seshadri constants of $K3$ surfaces of degrees $6$ and $8$}
\author{Concettina Galati }
\address{Dipartimento di Matematica\\
 Universit\ag\, della Calabria\\
via P. Bucci, cubo 31B\\
87036 Arcavacata di Rende (CS), Italy. }
\email{galati@mat.unical.it }
\author{Andreas Leopold Knutsen}
\address{Department of Mathematics, University of Bergen, Johs. Brunsgt. 12, N-5008 Bergen, Norway}
\email{andreas.knutsen@math.uib.no}

\thanks{The second author thanks
the Department of Mathematics at the University of Calabria for its kind hospitality. The first author was supported by GNSAGA of INdAM
and by the PRIN 2008 {\it Geometria delle variet\ag\,  algebriche e dei loro spazi di moduli}, co-financed by MIUR}




\dedicatory{}

\commby{}


\begin{abstract}
We compute Seshadri constants $\eps(X):= \eps(\O_X(1))$ on $K3$ surfaces $X$ of degrees $6$ and $8$. 
Moreover, more generally, we prove that if $X$ is any embedded  $K3$ surface 
of degree $2r-2 \geq 8$ in $\PP^r$ not containing lines, then $1 < \eps(X) <2$ if and only if the 
homogeneous ideal of $X$ is not generated by only quadrics 
(in which case $\eps(X)=\frac{3}{2}$).
\end{abstract}


\maketitle

\section{Introduction and results} \label{sec:intro}

In the past couple of decades there has been considerable interest in studying the local positivity of nef line bundles on algebraic varieties. Demailly \cite{de} introduced Seshadri constants to capture the concept of local positivity.
Let $X$ be a smooth projective variety and $L$ a nef line bundle
on $X$. For a point $x\in X$ the real number
\[\eps(L,x):=\inf_{C \ni x}\frac{L. C}{\mult_xC}\]
 is the \emph{Seshadri constant of $L$ at $x$}. Here the
   infimum is taken over all curves $C$ passing through
   $x$ and it is easily seen that one can restrict to considering only irreducible curves. Equivalently,
\[\eps(L,x):=\sup\{\eps \in \RR \; | \; f^*L-\eps E \; \; \mbox{is nef}\},\]
where $f:\widetilde{X} \to X$ is the blow up at $x$ and $E \subset \widetilde{X}$ is the exceptional divisor. 

The \emph{(global) Seshadri constant of $L$} is defined as 
\[ \eps(L):=\inf_{x \in X} \eps(L,x). \]

One has $\eps(L) >0$ if and only if $L$ is ample, by Seshadri's criterion. Moreover, 
$\eps(L) \leq \sqrt[n]{L^n}$, where $n:=\dim X$, by Kleiman's theorem \cite[Rem.~1.8]{EKL}.

Using standard terminology, we say that a curve $C$ is a {\it Seshadri curve (of $L$)} if $C$ computes 
$\eps(L)$, that is, if $\eps(L)=\frac{L. C}{\mult_xC}$ for a point $x\in C$.

We refer to \cite{ba-sur,PAG,prim}, to mention a few, for accounts on Seshadri constants and the development in the research on them.

One of the very subtle points about Seshadri constants is that their values are only known in few cases. For instance, all known examples are rational and it is not even known whether  Seshadri constants are always rational or not. 

If $X$ is a $K3$ surface and $L$ is ample, then the following is known:
\begin{itemize}
\item If $L$ is not globally generated, then $\eps(L)=\frac{1}{2}$. Otherwise, $\eps(L)\geq 1$
\cite[Prop.~3.1]{BDS}. (Indeed, by the nowadays classical results of Saint-Donat \cite{SD}, an ample line bundle $L$ on a $K3$ is not globally generated if and only if there is an elliptic pencil $|E|$ such that $E.L=1$. In this case, the Seshadri curves are the curves in $|E|$ with a double point.)
\item If $L$ is globally generated but not very ample, then $\eps(L)=1$ or $2$. This also follows from \cite{SD} and is probably well-known to the experts, but we give a proof in Proposition \ref{prop:notva} in \S\;\ref{sec:proof} for lack of a reference.  
\item If $L$ is very ample, then $\eps(L)=1$ if and only if $X$ embedded by $|L|$ contains a line \cite[Thm.~2.1(a)]{ba-sur}. (This holds on any surface.) Moreover, it is well-known that $K3$ surfaces containing a line form a codimension one subset in the moduli space of polarized
$K3$ surfaces of fixed degree.
\item If $L$ is very ample with $L^2=4$, that is, $X$ embedded by $|L|$ is a smooth quartic surface in $\PP^3$, then $\eps(L)=1$, $\frac{4}{3}$ or $2$ and all three cases occur \cite{ba-qua}: $\eps(L)=1$ if and only if $X$ contains a line; $\eps(L)=\frac{4}{3}$ if and only if $X$ contains no line and contains a rational curve $C \in |L|$ with a triple point (which happens if and only if $X$ contains a point where the Hesse form vanishes); $\eps(L)=2$ in all other cases.
Furthermore, the two first cases occur on codimension one subsets in the space of quartic surfaces.
\item If $\Pic X \cong \ZZ[L]$ and $L^2$ is a square, then $\eps(L)=\sqrt{L^2}$ \cite{kn-ses}.
\end{itemize}

By the above, when studying Seshadri constants on a $K3$ surface $X$, one can without 
loss of generality assume that $L$ is very ample, whence that the complete linear system $|L|$
embeds $X$ as a surface of degree $L^2=2r-2$ in $\PP^r$, where $r:=\dim |L|$. We set $\eps(L)=\eps(X)$. Moreover, we can also assume that $X$ does not contain any lines, that is, that $\eps(X) >1$. 

In this note we will compute the Seshadri constant $\eps(X)$ of $X$ if $L^2=6$ or $8$.
It is well-known that the homogeneous ideal of a $K3$ surface $X$ of degree $2r-2$ in $\PP^r$ is always generated by quadrics and cubics \cite{SD}. In particular, $X$ is a complete intersection of type $(2,3)$ in $\PP^4,$  that is, a complete intersection of a quadric and a cubic hypersurface if $L^2=6$. By contrast, if $L^2=8$, there are two types of projective models. In general $X$ is a complete intersection of type $(2,2,2)$ in $\PP^5$, that is, of three hyperquadrics. In addition, there is a codimension one subspace of polarized $K3$ surfaces of degree $8$ that are  a section of $|\O_Y(3)-\F|$ in a smooth three-dimensional rational normal scroll $Y$ in $\PP^5$, where $\F$ denotes the class of the $\PP^2$-fibers \cite{SD,JK}. The homogeneous ideal of these cannot be generated only by quadrics. In this latter case, $\eps(L)=\frac{3}{2}$, by the following more general result, which we will prove in \S\;\ref{sec:proof}.

\begin{proposition} \label{thm:notquad}
  Let $X \subset \PP^r$ be a  $K3$ surface of degree $2r-2$ not containing lines, $r \geq 5$.
Then $1< \eps(X) <2$ if and only if the homogeneous ideal of $X$ is not generated only by quadrics, in which case
$\eps(X) =\frac{3}{2}$ and the irreducible Seshadri curves are the irreducible curves with a double point in an elliptic pencil of degree $3$. 
\end{proposition}

Now to our main result, whose proof is also contained
in  \S\;\ref{sec:proof}.

\begin{thm} \label{thm:deg68}
  Let $X$ be a  $K3$ surface not containing any lines. 

(a) If $X \subset \PP^4$ is a surface of type $(2,3)$, then one of the following cases occurs:
\begin{itemize}
\item[(a-i)]  $\eps(X)=2$, and there is a Seshadri curve that is a hyperplane section with a triple point;
\item[(a-ii)]  $\eps(X)=\frac{3}{2}$, and all irreducible Seshadri curves are irreducible curves with a double point in an elliptic pencil of degree $3$;
\end{itemize}

(b) If $X\subset \PP^5$ is a surface of type $(2,2,2)$, then one of the following cases occurs:
\begin{itemize}
\item[(b-i)]  $\eps(X)=\frac{8}{3}$,  and there is a Seshadri curve that is a hyperplane section with a triple point;
\item[(b-ii)]  $\eps(X)=\frac{5}{2}$, and all irreducible Seshadri curves are irreducible curves with a double point in an elliptic pencil of degree $5$;
\item[(b-iii)]  $\eps(X)=2$, and all irreducible Seshadri curves are either conics or  irreducible curves with a double point in an elliptic pencil of degree $4$.
\end{itemize}

Furthermore, cases (a-ii), (b-ii) and (b-iii) occur on sets of codimension one in the space of all such surfaces. 
\end{thm}

The proof  is based on existence results of curves with triple points on a $K3$ surface 
in \cite{ga} and in Section \ref{sec:existence} of this paper, combined with properties of curves on $K3$ surfaces. 
In \S\;\ref{sec:max} we make some comments on how far similar reasonings for curves with singularities of higher multiplicities would bring us.

\vspace{0.3cm}
\noindent {\bf Conventions.}
  We work over the field of complex numbers. 
We will assume familiarity with standard results on $K3$ surfaces, like the results in \cite{SD} (a brief summary of those can be found in \cite[\S\;2]{kn-sm}), and about deformations and families of $K3$ surfaces, as explained for instance in \cite{BPV,GH,ko}.

\section{Curves with a triple point on a $K3$ surface of degree $6$}\label{sec:existence}

The main ingredient in Theorem \ref{thm:deg68} is the existence of irreducible curves in
 the linear system $|\mathcal O_X(1)|$ with a singularity of multiplicity $3$ at a (special) point of $X,$
where $X$ is a general $K3$  surface of degree $2r-2$ in $\PP^r,$ with  $r=4,5.$ For $r=5$ this follows by 
\cite[Cor. 4.3]{ga}. The argument used in \cite[Section 4]{ga} does not apply to the case $r=4$ (see \cite[Rmk. 4.4]{ga}).
The existence of elliptic curves in $|\mathcal O_X(1)|$ 
with a triple point  will be proved in this section.  
We need to recall some classical deformation theory of $K3$  surfaces (cf. \cite{clm}). 

Let $R_1$ and $R_2,$ with $R_1\simeq R_2\simeq \mathbb F_1,$ be two general rational normal scrolls 
of degree $3$ in $\PP^4,$ generated by the union of the secants to divisors in two general $g^1_2$'s on 
an elliptic normal curve $E\subset\mathbb P^4$ of degree $5.$ Then $E=R_1\cap R_2$ 
and the intersection is transverse. Moreover, if $\sigma_i$ and $F_i$ are the two generators
of $\Pic R_i$, with $\sigma^{2}_i=-1$ and $F_i^2=0,$ we have that 
$|\mathcal O_{R_i}(1)|=|\mathcal O_{R_i}(\sigma_i+2F_i)|,$
for $i=1,2.$ By classical deformation theory of $K3$  surfaces (cf. \cite[Cor. 1, Thms. 1 and 2]{clm}
and related references, precisely, \cite[Rmk. 2.6]{F} and \cite[\S 2]{GH85}), we know that, no matter how we choose
$16$ general points $\{\xi_1, \ldots, \xi_{16}\}$ on $E,$ there exists a smooth family of surfaces 
$\mathcal X\to\mathbb A^1$ whose general fibre $\mathcal X_t$ is a smooth $K3$  surface  of degree $6$ in $\mathbb P^4$
such that $\Pic(\mathcal X_t)\simeq \mathbb Z[\mathcal O_{\mathcal X_t}(1)]$ and whose special fibre 
is $\mathcal X_0=R_1\cup \tilde R_2,$ where $\tilde R_2$ is the blowing-up of $R_2$ at $\{\xi_1, \ldots, \xi_{16}\}$
with exceptional divisors $E_1,\ldots,E_{16}.$ We want to obtain curves with a triple point in $|\mathcal O_{\mathcal X_t}(1)|$
as deformations of suitable curves in $|\mathcal O_{\mathcal X_0}(1)|.$

\begin{lemma} \label{lemma:pre}
Let $R=R_1\cup R_2\subset \PP^4$ be the union of two general rational normal scrolls as above. Then there exists
a unique curve $C=C_1\cup C_2\in |\mathcal O_R(1)|$, where $C_i\subset R_i$, $i=1,2$, such that both $C_i$'s have a node at the same point of $E$ with one branch tangent to $E$. 
\end{lemma}

\begin{proof}
Assume that there exists a curve $C=C_1\cup C_2\in |\mathcal O_R(1)|$ as in the statement. Let $p \in E$ be the point where  both $C_i$'s have a node with one branch tangent to $E$. 
Then $C_i\subset R_i$ intersects
the unique divisor $D_i$ in $|F_i|$ passing through $p$ with multiplicity $2$. 
As $F_i.C_i=F_i.(\sigma_i+2F_i)=1$, we have $D_i\subset C_i.$ More precisely, we have  
 
\begin{enumerate}
\item $C_i=  D_i\cup L_i,$
 where $D_i\sim F_i$ and $L_i\sim \sigma_i+F_i,$ for $i=1,2$; \label{cond1}
\item $C\cap E=3p+q_1+q_2,$ where $p,q_1,q_{2}$ are distinct points satisfying
\begin{eqnarray*}
D_1\cap E\,\, (\mbox{resp.}\,\, D_2\cap E)& =& p+q_{2}\,\,(\mbox{resp.}\,\,  p+q_{1})\,\,\textrm{and}\\
L_1\cap E\,\,(\mbox{resp.}\,\, L_2\cap E ) &=& 2p+q_{1}\,\,
 (\mbox{resp.}\,\,2p+q_2). \nonumber
\end{eqnarray*}\label{cond2}
\end{enumerate}
\vspace{-0.6cm}
In particular, $C$ has nodes at $q_1$ and $q_2,$ a space quadruple point at $p$ as in the following figure and it is smooth
on $R\setminus E.$  
\begin{figure}[htbp] 
\includegraphics[width=6cm]{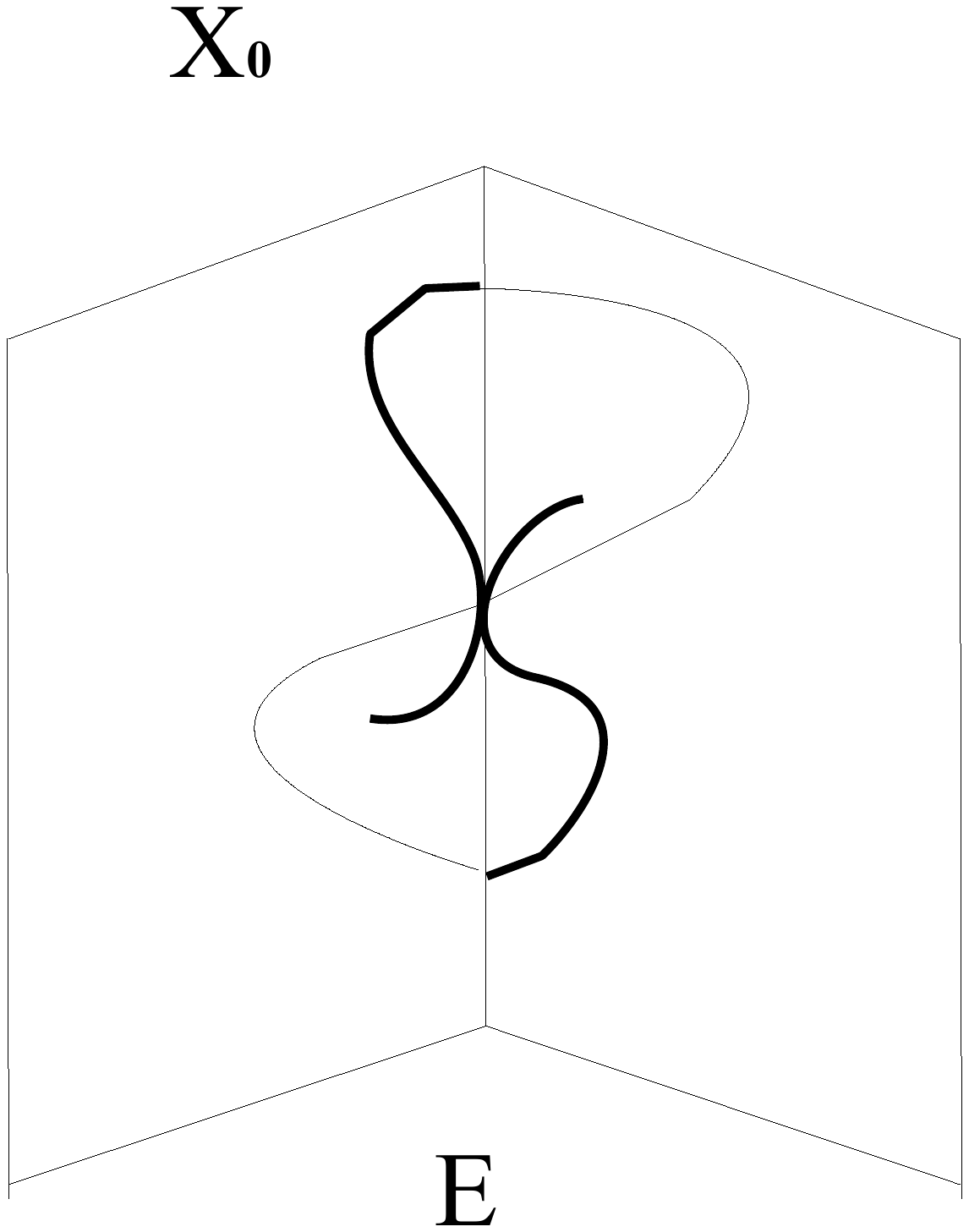}
\end{figure} 
It follows that the existence of a curve $C$ as in the statement is equivalent to the existence of a unique point 
$p\in E$ such that the unique divisor $L_1\sim \sigma_1+F_1$ on $R_1$ tangent to $E$ at $p$ satisfies $L_1\cap E=2p+q_1,$
where $q_1\neq p$ and $p+q_1\sim F_2$ on $E$. We claim that this is satisfied by the point $p$ that is the unique member in the degree one linear system
$|\O_E(\sigma_1+F_1-F_2)|$ on $E$. Indeed, if $D_2\sim F_2$ is the unique divisor on $R_2$
 passing through $p$, we may assume $D_2\cap E=p+q_1$ with $p\neq q_1$, by the generality of the linear system $|\O_E(F_2)|$ on $E$ (spanning the normal scroll $R_2$ in $\PP^4$).
 The unique divisor $L_1\in |\mathcal O_{R_1}(\sigma_1+F_1)|$ cutting $2p+q_1$ on $E$
 must be smooth by  a trivial argument. Let $D_1\in |\mathcal O_{R_1}(F_1)|$ be the unique divisor passing through $p$
 and let $p+q_2=D_1\cap E$, where $q_2$ is different from $p$ and $q_1$, because 
$D_1.L_1=1$. Since $D_1+L_1\in |\mathcal O_{R_1}(1)|,$ we have that $3p+q_1+q_2\in |\mathcal O_E(1)|$ and thus there 
 exists a unique smooth curve $L_2\sim \sigma_2+F_2$ on $R_2$ cutting $2p+q_2$ on $E.$ 
 This proves our claim, whence also the lemma.
\end{proof}

\begin{proposition}\label{existence}
Let $X$ be a general $K3$  surface of degree $6$ in $\PP^4$ such that $\Pic(X)\simeq\mathbb Z[\mathcal O_X(1)].$
Then in the linear system $|\mathcal O_X(1)|$ there exist finitely many elliptic curves with a singularity of multiplicity $3$.
\end{proposition}

\begin{proof}
Let $R_1$ and $R_2$ and $\mathcal X\to\mathbb A^1$ as above. In the
linear system $|\mathcal O_{\mathcal X_t}(1)|$ on the general fibre $\mathcal X_t$ of $\mathcal X\to\mathbb A^1$
there are no curves with a singularity of multiplicity bigger than $3$, because every such curve must be reducible
and $\Pic {\mathcal X_t} \cong \mathbb{Z}[\O_{\mathcal X_t}(1)]$. Moreover, if there exist curves with a triple point, then they are finitely many
by \cite[Lemma 3.1]{C}. We want to prove their existence.
By Lemma \ref{lemma:pre}, there exists a unique curve
$C=C_1\cup C_2\in |\mathcal O_{\mathcal X_0}(1)|$ with a space quadruple point 
at a point $p\in E$
as in the figure on the previous page, and $C$ must satisfy conditions \eqref{cond1} and \eqref{cond2}.
We will show that $C$ may be deformed to a curve on $\mathcal X_t$ in such a way that
the non-planar quadruple point of $C$ at $p$ is deformed to a triple point.  
The scheme parametrizing deformations of $C$ in $\mathcal X$ is an irreducible component 
$\mathcal H$ of the relative Hilbert scheme $\mathcal H^{\mathcal X|\mathbb A^1}$ of the family $\mathcal X.$
The scheme parametrizing deformations of the singularity of $C$ at $p$ is the versal deformation space
$T^1_{C,p}.$ By \cite[Lemma 3.2]{ga}, we may choose local coordinates $x,y,z,t$ of $\mathcal X$ centered at $p$
such that $C$ is given by the equations
\begin{eqnarray}\label{eq: quadruple-point}
\left\{\begin{array}{l}
(y+x+z^2)z=0,\\
xy=t, \\
t=0\\
\end{array}\right.
\end{eqnarray}
at $p$. Moreover, by  \cite[(10) in proof of Lemma 3.2]{ga},
the versal deformation family $\mathcal C_p\to T^1_{C,p}$ is given by the equations
\begin{eqnarray}\label{versal_family}
\left\{\begin{array}{l}
(y+x+z^2)z+a_1+a_2x+a_3y+a_4z=0,\\
xy+b_1+b_2z+b_3z^2=0,\\
\end{array}\right.
\end{eqnarray}
where $(a_1,\ldots,a_4,b_1,b_2,b_3)=(\underline a,\underline b)$ are affine coordinates on $T^1_{C,p}.$
In particular, we have that $T^1_{C,p}\simeq \mathbb C^7.$
Denote by $\mathcal D\to\mathcal H$ the universal family parametrized by $\mathcal H$. By versality, there exist analytic
 neighborhoods $U_p$ of $[C]$ in $\mathcal H,$ $U_p^\prime$ of $p$ in $\mathcal D$ and  $V_p$ of $\underline 0$ in $T^1_{C,q}$
 and a map $\phi_p:U_p\to V_p$ 
so that the family
$\mathcal D|_{U_p}\cap U_p^\prime$ is isomorphic to the pull-back of $\mathcal C_p|_{V_p}$, with respect to $\phi_p$,
\begin{equation}\label{versalmap}
\xymatrix{
\mathcal C_p  \ar[d] & \mathcal C_p |_{V_p} \ar[d]  \ar@{_{(}->}[l]& 
U_p\times_{V_p}\mathcal C_p |_{V_p}  \ar[l] \ar[r]^{\hspace{0.5cm}\simeq} \ar[dr] & \mathcal D|_{U_p}\cap U_p^\prime
\ar@{^{(}->}[r] \ar[d] & \mathcal D  \ar[d] \\
T^1_{C,p} &  V_p \ar@{_{(}->}[l] & & \ar[ll]_{\phi_p} U_p \ar@{^{(}->}[r] & \mathcal H.}
\end{equation} 
 
We need to describe the image $\phi_p(U_p)\subset V_p.$ First observe that
the map $\phi_p$ is finite. To see this, observe that the fibre $\phi_p^{-1}(\underline 0)$
over $\underline 0$ is the analytically equisingular deformation locus of $[C]$ in $U_p.$
Since $C$ has a space singularity, all equisingular deformations of $C$ are trivially contained in $\mathcal X_0.$
By the previous lemma, we deduce that $\phi_p^{-1}(\underline 0)$ consists of the unique point $[C],$ possibly with multiplicity.
 Thus $\phi_p(U_p)\subset T^1_{C,p}$ is an irreducible analytic subvariety of dimension 
$5=\dim|\mathcal O_{\mathcal X}(1)|=\dim|\mathcal O_{\PP^4}(1)|+1,$
containing $\phi_p(U_p\cap |\mathcal O_{\mathcal X_0}(1)|)$
as a codimension $1$ subvariety.
Moreover, using that the dimension of $\phi_p(U_p\cap |\mathcal O_{\mathcal X_0}(1)|)$ is $4$ and that it
parametrizes all deformations of the curve corresponding to $\phi_p([C])$ on the surface $xy=0,$ we have that  $\phi_p(U_p\cap |\mathcal O_{\mathcal X_0}(1)|)$
 is given by the equations $b_1= b_2 = b_3=0.$
We denote by $\mathcal T$ the subvariety of $T^1_{C,p}$ defined as the Zariski closure of 
the locus of points 
$(\underline a,\underline b)$ that correspond to curves with
an ordinary singularity of multiplicity $3$ at a point $(x_0,y_0,z_0)\neq (0,0,z_0)$ of a smooth surface $xy+b_1+b_2z+b_3z^2=0.$
By \cite[proof of Thm. 3.9]{ga} we know that $\mathcal T$ is nonempty and  intersects the  
locus $b_2=b_3=0$ along the curve $\gamma:a_1=a_2=a_3=b_2=b_3=4b_1-a_4^2=0.$
In particular, we have that 
$\phi_p(U_p\cap |\mathcal O_{\mathcal X_0}(1)|)\cap \mathcal T=\underline 0$.
It follows that $\dim(\phi_p(U_p)\cap \mathcal T) \leq 1$, and we claim that equality is attained.

To see this, we first note that every irreducible component of $\mathcal T$ has codimension at most $4$ in $T^1_{C,p}$. (Indeed, using that $(x_0,y_0)\neq (0,0)$, we can recover $x$ or $y$ from the second equation in  \eqref{versal_family}. Substituting into the first, we get the equation of a planar curve. Imposing that this curve has a triple point, one obtains four equations defining $\mathcal T$.) 
It follows that $\dim(\phi_p(U_p)\cap \mathcal T)\geq 5+3-7=1$, proving that
$\dim(\phi_p(U_p)\cap \mathcal T) = 1$, as desired.

  This finishes the proof of the proposition, by versality.
\end{proof}

\section{Computing Seshadri constants} \label{sec:proof}

We start with the following result mentioned in the introduction:

\begin{proposition} \label{prop:notva}
  Let $L$ be an ample, globally generated line bundle on a $K3$ surface $X$. Then $L$ is not very ample if and only if one of the following occurs:
  \begin{itemize}
  \item[(a)] $L^2=2$;
  \item[(b)] there is an elliptic pencil $|E|$ such that $E.L=2$;
   \item[(c)] $L \sim 2B$, with $B$ globally generated such that $B^2=2$. 
  \end{itemize}
Furthermore, we have $\eps(L)=1$ in cases (a)-(b) and $\eps(L)=2$ in case (c).
\end{proposition}

\begin{proof}
  The characterization of $L$ as in (a)-(c) follows from \cite{SD}. Moreover, 
$\eps(L) \geq 1$ as $L$ is globally generated and ample \cite[Prop.~3.1]{BDS}.

 It is well-known that any elliptic pencil on a $K3$ surface contains singular members. Thus $\eps(L)=1$ in case (b). 

In case (a), the linear system $|L|$ defines a double cover of $X$ onto $\PP^2$ branched along a reduced plane sextic. The pullback $C$ of any line tangent to 
the sextic has a point of multiplicity two, whence $\eps(L) =1$. It also follows that $\eps(L)=2\eps(B)=2$ in case (c).
\end{proof}

We recall the following fact: If $\eps(L) < \sqrt{L^2}$, there is an irreducible  Seshadri curve  $C$ on $S$, that is,  an irreducible curve $C$ and a point $x \in C$ such that $\eps(L)=\frac{L.C}{\mult_xC}$, cf. \cite[Lemma 2.1]{og} or \cite[Lemma 3.1]{sz}. For any curve $C$ we set 
\[ \eps_{C,x}:= \frac{L.C}{\mult_xC} \; \; \mbox{and} \; \; m_x:=\mult_xC. \]
If $C$ is irreducible, then, as $p_a(C)-p_g(C) \geq \frac{1}{2}m_x(m_x-1)$, we obtain by adjunction on a $K3$ surface that
\begin{equation}
  \label{eq:1}
  C^2 \geq m_x(m_x-1)-2.
\end{equation}

We will  repeatedly make use of the following well-known fact:
Any linear system $|D|$ on a $K3$ surface satisfying $D^2 \geq -2$ and $D.H >0$ for some ample divisor $H$ is nonempty. If in addition $D^2 \geq 0$, then $|D|$ contains a singular member, that is, an element (possibly nonreduced or reducible) with a point of multiplicity $\geq 2$.

\renewcommand{\proofname}{Proof of Proposition \ref{thm:notquad}}

\begin{proof}
  Assume that $1< \eps(X)<2$ and set $L:=\O_X(1)$. As $2 < \sqrt{L^2}$, there  
exists an irreducible  Seshadri curve $C$ with $x \in C$ such that 
$\eps_{C,x}=\eps(X)$. Then $2 \leq m_x < L.C <2m_x$. Hence, by \eqref{eq:1} and the Hodge index theorem,
\[L^2 \Big(m_x(m_x-1)-2\Big) \leq L^2C^2  \leq (L.C)^2 < 4m_x^2. \]
Using the facts that $L^2 \geq 8$ and $m_x \geq 2$,
one easily verifies that the only possibility is
$(m_x,C^2,L.C)=(2,0,3)$, whence $\eps(X) =\frac{3}{2}$. Furthermore, the existence of $C$ implies that the homogeneous ideal of $X$ is not generated only by quadrics by \cite{SD}. 

Conversely, assume that the homogeneous ideal of $X$ is not generated only by quadrics. Then by \cite{SD} there exists an elliptic pencil $|E|$ such that $E.L=3$. As above, any member $C \in |E|$  with a point $x$ of multiplicity two satisfies $\eps_{C,x}=\frac{3}{2}$, whence
$\eps(L) \leq \frac{3}{2}$. Furthermore, $\eps(L)>1$ as $X$ does not contain lines.
\end{proof}

\renewcommand{\proofname}{Proof of Theorem \ref{thm:deg68}}

\begin{proof}
  We set $L:=\O_X(1)$. 

We first treat the case $X=(2,3) \subset \PP^4$, where $L^2=6$.

By Proposition \ref{existence}, on the general polarized $K3$ surface $(X,L)$ of degree $6$ there is an irreducible curve $C \in |L|$ with a triple point $x$. Hence $\eps(L) \leq \eps_{C,x}=\frac{6}{3}=2$. By lower semi-continuity of Seshadri constants in a family of smooth surfaces \cite[Cor.~5]{og}, we have $\eps(L) \leq 2$ for any $(X,L)$. Furthermore, if equality holds, there is a hyperplane section with a triple point that is a Seshadri curve. 

Assume now that $\eps(L)<2$. As $2 < \sqrt{6}$, there must exist an irreducible  Seshadri curve $C$ with $x \in C$ such that $\eps_{C,x}=\eps(L)<2$. From \eqref{eq:1} and the Hodge index theorem, we have
\begin{equation}
  \label{eq:2}
 6 \Big(m_x(m_x-1)-2\Big) \leq 6C^2  \leq (L.C)^2 < 4m_x^2. 
\end{equation}
If $m_x \geq 3$, this can only be satisfied if $(m_x,L.C,C^2)=(3,5,4)$. But then $D:=L-C$ satisfies $D^2=0$ and $L.D=1$, contradicting the fact that $L$ is globally generated, by \cite{SD}.

Since $m_x \geq 2$, as $X$ is assumed not to contain lines, we are left with the case
 $m_x =2$, where \eqref{eq:2} yields $C^2=0$ and $L.C=3$ (using \cite{SD} and the fact that $L$ is very ample). This means that $C$ is a singular member of an elliptic pencil. By e.g. \cite[Thm.~1.1]{kn-sm} there exist $K3$ surfaces of degree $6$ with the Picard group generated by the hyperplane section and an elliptic degree $3$ curve. This occurs on sets of codimension one in the space of all such surfaces, cf. \cite[Thm.~14]{ko} or \cite[p.~594]{GH}. Hence the theorem follows in this case.

We next treat the case  $X=(2,2,2) \subset \PP^5$, where $L^2=8$.

By \cite[Thm.~1.1]{ga}, on the general polarized $K3$ surface $(X,L)$ of degree $8$ there is an irreducible curve $C \in |L|$ with a triple point $x$. Hence $\eps(L) \leq \eps_{C,x}=\frac{8}{3}$. By lower semi-continuity again, $\eps(L) \leq \frac{8}{3}$ for any $(X,L)$, and, if equality holds, there is a hyperplane section with a triple point that is a Seshadri curve. 

Assume now that $\eps(L)<\frac{8}{3}$. As $\frac{8}{3} < \sqrt{8}$, there must exist an irreducible   Seshadri curve $C$ with $x \in C$ such that $\eps_{C,x}=\eps(L)<\frac{8}{3}$. From \eqref{eq:1} and the Hodge index theorem, 
\begin{equation}
  \label{eq:3}
 8 \Big(m_x(m_x-1)-2\Big) \leq 8C^2  \leq (L.C)^2 < \frac{64}{9}m_x^2. 
\end{equation}
One easily checks that this cannot be satisfied for $m_x \geq 9$.

If $m_x=8$, the only solution to \eqref{eq:3} is $(L.C,C^2)=(21,54)$. Set $D:=3L-C$. Then $(D^2,L.D)=(0,3)$, so any curve $D_0 \in |D|$ with a singular point $y$ satisfies $\eps_{D_0,y} \leq \frac{3}{2} < \frac{21}{8}= \eps_{C,x}=\eps(L)$, a contradiction.

If $m_x=7$, the only solution to \eqref{eq:3} is $(L.C,C^2)=(18,40)$. Set $D:=C-2L$. Then $(D^2,L.D)=(0,2)$, contradicting the very ampleness of $L$ by \cite{SD}.

If $m_x=6$, the only solution to \eqref{eq:3} is $(L.C,C^2)=(15,28)$. Set $D:=2L-C$. Then $(D^2,L.D)=(0,1)$, contradicting the global generation  of $L$ by \cite{SD}.

If $m_x=5$, the only solutions to \eqref{eq:3} are $(L.C,C^2)=(12,18)$, $(13,18)$ and $(13,20)$.
Set $D:=C-L$. Then $(D^2,L.D)=(2,4)$, $(0,5)$ and $(2,5)$, respectively. But then any curve $D_0 \in |D|$ with a singular point $y$ satisfies $\eps_{D_0,y} \leq \frac{L.D}{2} < \frac{L.C}{5}=\eps_{C,x}= \eps(L)$, a contradiction.

If $m_x=4$, the only solutions to \eqref{eq:3} are $(L.C,C^2)=(9,10)$, $(10,10)$ and 
$(10,12)$. As above, set $D:=C-L$. Then $(D^2,L.D)=(0,1)$, $(-2,2)$ and $(0,2)$, respectively. 
The first and third case contradict global generation and very ampleness, respectively, of $L$ by \cite{SD}. The middle case yields the contradiction $\eps(L) \leq L.D=2 < \frac{5}{2} =\eps_{C,x}$.

If $m_x=3$, the only solutions to \eqref{eq:3} are $(L.C,C^2)=(6,4)$, $(7,4)$ and 
$(7,6)$. Then $D:=L-C$ satisfies $(D^2,L.D)=(0,2)$, $(-2,1)$ and $(0,1)$, respectively, yielding the same contradictions as in the previous case.

If $m_x=2$ and $C^2 \geq 2$, the only solutions to \eqref{eq:3} are $(L.C,C^2)=(4,2)$ and 
$(5,2)$. In the first case the Hodge index theorem yields $L \sim 2C$, contradicting the very ampleness of $L$ by \cite{SD}. In the second case $D:=L-C$ satisfies $(D^2,L.D)=(0,3)$
contradicting the fact that the homogeneous ideal of $X$ is generated only by quadrics by \cite{SD}. 

If $m_x=2$ and $C^2=0$, then by \eqref{eq:3} we must have $L.C \leq 5$. Since $L$ is very ample and the homogeneous ideal of $X$ is generated only by quadrics, we must have $L.C \geq 4$ by \cite{SD}. Finally, if $m_x=1$, then \eqref{eq:3} yields $L.C \leq 2$, whence $C^2=-2$ by \cite{SD} as $L$ is very ample and $C$ is irreducible. Thus $C \cong \PP^1$. As $X$ is assumed not to contain lines, we must have $L.C=2$.

To summarize, we have proved that if $\eps(L) < \frac{8}{3}$, then the only possible irreducible Seshadri curves satisfy $(L.C,C^2,m_x)=(5,0,2),(4,0,2),(2,-2,1)$.

In any of the above cases, the divisors $L$ and $C$ are linearly independent in $\Pic X$. There are complete intersections $X$ of type $(2,2,2)$ in $\PP^5$ such that $\Pic X \cong \ZZ[L] \oplus \ZZ[C]$, with $L = \O_X(1)$ and $C$ as in any of the above cases, cf. \cite[Thms. 1.1 and 6.1]{kn-sm}. One can verify that in each of the  cases there are no divisors with the intersection properties of the two other cases,
whence $\eps(X)=\eps_{C,x}$. Therefore, cases (b-ii) and (b-iii) do occur, and they occur in codimension one subsets of the space of all such surfaces by \cite[Thm.~14]{ko} or \cite[p.~594]{GH}. This finishes the proof of the theorem.
\end{proof}

\section{Curves with points of maximal expected multiplicities} \label{sec:max}

In a nonempty linear system $|D|$ on a smooth projective surface $X$ one expects the existence of curves with a point of multiplicity $m$ for any $m$ satisfying
\begin{equation} \label{eq:exp}
  \dim |D|-\frac{1}{2}m(m+1)+2 \geq 0.
\end{equation}
We therefore make the following definition.

\begin{definition} \label{def:maxexp}
  A curve $C$ in a linear system $|D|$ on a smooth projective surface $X$ is said to have a point of {\it maximal expected multiplicity}  if it has a point of multiplicity $m_0$, where $m_0$ is the maximal integer $m$ satisfying \eqref{eq:exp}. 
\end{definition}

We note that the curves with a triple point constructed in \cite{ga} and in Proposition \ref{existence} that we used in the proof of Theorem \ref{thm:deg68} are curves in the hyperplane system $|L|$ with a point of maximal expected multiplicity. The reason why we were able to compute all Seshadri constants was the fact that these curves induced a Seshadri constant $<\sqrt{L^2}$. Any lower Seshadri constant must then be induced by an irreducible Seshadri curve and the rest of the proof is merely to find out the different intersection properties of these. The same approach would of course also work if the curves with a point of maximal expected multiplicity were elements of $|nL|$ for $n>1$, and also if the maximal expected multiplicity is $>3$.  The argument in \cite{ga} to prove the existence of curves with a triple point can in principle be generalized to curves with singularities of multiplicities $m>3$. 
One may therefore ask how far this procedure will reach if one generalizes \cite{ga}. The next proposition shows, however, that there is little to win: one can at the best use this procedure to compute Seshadri constants on $K3$ surfaces of degrees $14$ and $24$.

\renewcommand{\proofname}{Proof}

\begin{proposition}
  Let $L$ be a globally generated line bundle on a $K3$ surface $X$ such that $L^2 \geq 4$. Assume that there is a curve $C \in |nL|$, for some $n \in \ZZ^+$, with a point $x$ of maximal expected multiplicity $m$ such that $\eps_{C,x} < \sqrt{L^2}$. Then
\[ (L^2,n,m) \in \{(6,1,3),(6,2,5),(8,1,3),(14,1,4),(24,1,5)\}. \]
\end{proposition}

\begin{proof}
  From \eqref{eq:exp} we have
  \[
  \frac{1}{2}m(m+1)-2\leq \dim |nL|= \frac{1}{2}n^2L^2+1.  
  \]
By assumption, $\eps_{C,x}=\frac{L.C}{m} =\frac{nL^2}{m} < \sqrt{L^2}$, whence
$n^2L^2 <m^2$, so that $m \leq n^2L^2-m^2+6 \leq 5$.
As $L^2 \geq 4$, the only solutions are the ones listed.
\end{proof}

We remark that also in the case of quartic surfaces \cite{ba-qua}, curves with a point of maximal expected multiplicity are not enough to compute all Seshadri constants: in fact, in the 
case $\eps(L)=\frac{4}{3}$ the irreducible Seshadri curves are rational curves $C \in |L|$ with a triple point, whereas the maximal expected multiplicity is $2$. However, this does not happen on the general quartic, so hyperplane sections with a point of maximal expected multiplicity $2$ are Seshadri curves yielding $\eps(X)=2$ in the general case.

\end{document}